\documentclass[10pt,a4paper]{article}
\usepackage{epsfig,amssymb,latexsym,amsthm}
\usepackage{amsmath}

\newtheorem{lemm}{Lemma}
\newtheorem{defi}{Definition}
\newtheorem{theo}{Theorem}
\newtheorem{rem}{Remark}
\newtheorem{cor}{Corollary}
\newtheorem{propo}{Proposition}

\newenvironment{definition}{\begin{defi}}{\end{defi}}
\newenvironment{lemma}{\begin{lemm}}{\end{lemm}}
\newenvironment{theorem}{\begin{theo}}{\end{theo}}
\newenvironment{remark}{\begin{rem}}{\end{rem}}
\newenvironment{corollar}{\begin{cor}}{\end{cor}}
\newenvironment{prop}{\begin{propo}}{\end{propo}}

\newcommand{\bee}{\begin{equation}}
\newcommand{\ee}{\end{equation}}
\newcommand{\mb}{\mathbb}
\newcommand{\mf}{\mathbf}
\newcommand{\ml}{\mathcal}
\newcommand{\ds}{\displaystyle}
\newcommand{\bs}{\boldsymbol}
\newcommand{\ul}{\underline}

\begin{document}
\title{Projectability of stable, partially free $\ml H$-surfaces in the non-perpendicular case}
%\titlerunning{On the regularity of $H$-surfaces near smooth parts of the free boundary}
\author{Frank M\"uller\footnote{Frank M\"uller, Fakult\"at f\"ur Mathematik, Universit\"at Duisburg-Essen, 45117 Essen, e-mail: frank.mueller@uni-due.de}}
%\authorrunning{F.\,M\"uller}

\maketitle

\begin{abstract}
\noindent 
A projectability result is proved for surfaces of prescribed mean curvature (shortly called $\ml H$-surfaces) spanned in a partially free boundary configuration. Hereby, the $\ml H$-surface is allowed to meet the support surface along its free trace non-perpendicularly. The main result generalizes known theorems due to Hildebrandt-Sauvigny and the author himself and is in the spirit of the well known projectability theorems due to Rad\'o and Kneser. A uniqueness and an existence result are included as corollaries.

\vspace{2ex}
\end{abstract}

\section{Introduction}
Let us write $B^+:=\{w=(u,v)=u+iv\,:\ |w|<1,\ v>0\}$ for the upper unit half disc in the plane. Its boundary is divided into
	$$\partial B^+=I\cup J,\quad I:=(-1,1),\quad J:=\partial B^+\setminus I=\{w\in\overline{B^+}\,:\ |w|=1\}.$$
In the present paper, a \emph{surface of prescribed mean curvature $\ml H=\ml H(\mf p)\in C^0(\mb R^3,\mb R)$} or, shortly, an \emph{$\ml H$-surface} is a mapping $\mf x=\mf x(w):B^+\to\mb R^3\in C^2(B^+,\mb R^3)$, which solves the system
	\bee\label{g0.8}
	\begin{array}{l}
	\Delta\mf x=2\ml H(\mf x)\mf x_u\wedge\mf x_v\quad\mbox{in}\ B^+,\\[0.5ex]
	|\mf x_u|=|\mf x_v|,\quad\mf x_u\cdot\mf x_v=0\quad\mbox{in}\ B^+.
	\end{array}
	\ee
Here, $\mf y\wedge\mf z$ and $\mf y\cdot\mf z$ denote the cross product and the standard scalar product in $\mb R^3$, respectively.

Observe that an $\ml H$-surface is not supposed to be a regular surface, that means, it may possess \emph{branch points} $w_0\in B^+$ with $\mf x_u\wedge\mf x_v(w_0)=\bs0$. 

We consider $\ml H$-surfaces spanned in a projectable, partially free boundary configuration, which means the following:

\begin{definition}{\bf(Projectable boundary configuration)}\label{d1}\\
Let $ S=\Sigma\times\mb R\subset\mb R^3$ be an embedded cylinder surface over the planar closed Jordan arc $\Sigma=\pi( S)$ of class $C^3$; here $\pi$ denotes the orthogonal projection onto the $x^1,x^2$-plane. Furthermore, let $\Gamma\subset\mb R^3$ be a closed Jordan arc which can be represented as a $C^3$-graph over the planar closed $C^3$-Jordan arc $\underline\Gamma=\pi(\Gamma)$. Finally, assume $\underline\Gamma\cap\Sigma=\{\pi_1,\pi_2\}$, where $\pi_1,\pi_2$ are the distinct end points of $\underline\Gamma$ as well as $\Sigma$, and $\Gamma$ and $ S$ meet with a positive angle at the respective points $\mf p_1,\mf p_2\in\Gamma\cap S$ correlated by $\pi_j=\pi(\mf p_j)$, $j=1,2$. 
Then we call $\{\Gamma, S\}$ a \emph{projectable (partially free) boundary configuration}.
\end{definition}

To be precise, in Definition\,\ref{d1}, the phrase ''$\Gamma$ and $ S$ meet with a positive angle at the respective points $\mf p_1,\mf p_2\in\Gamma\cap S$'' means that the tangentential vector of $\Gamma$ is not an element of the tangential plane of $S$ at these points.

A \emph{partially free $\ml H$-surface} is a solution $\mf x\in C^2(B^+,\mb R^3)\cap C^0(\overline{B^+},\mb R^3)$ of (\ref{g0.8}), which satisfies the  boundary conditions
	\bee\label{g0.2}
	\begin{array}{l}
	\mf x(w)\in S\quad\mbox{for all}\ w\in I,\\[0.5ex]
	\mf x|_{ J}: J\to\Gamma\ \mbox{strictly monotonic},\\[0.5ex]
	\mf x(-1)=\mf p_1,\ \mf x(+1)=\mf p_2
	\end{array}
	\ee
for a given projectable boundary configuration $\{\Gamma, S\}$. Roughly speaking, we aim to show that any such partially free $\ml H$-surface is itself projectable. This is in the spirit of the famous projectability result for minimal surfaces by Rad\'o and Kneser and will be proved under additional assumptions on the $\ml H$-surface and the configuration $\{\Gamma, S\}$, namely: The boundary configuration shall be \emph{$R$-admissible} in the sense of Definition\,\ref{d2} below and the $\ml H$-surface shall be H\"older-continuous on $\overline{B^+}$, stationary w.r.t.~some energy functional $E_{\mf Q}$ and stable w.r.t.~the corresponding generalized area functional $A_{\mf Q}$. Here $\mf Q$ is a given vector field which satisfies a natural smallness condition and which possesses a suitable normal component w.r.t.~$ S$ as well as the divergence $\mbox{div}\,\mf Q=2\ml H$; see Section\,2 for details.

The first results of this type were given by Hildebrandt-Sauvigny \cite{hildefritz1}-\cite{hildefritz}. They considered the special case of minimal surfaces; a generalization to $F$-minimal surfaces can be found in \cite{mueller-winklmann}. Concerning partially free $\ml H$-surfaces the only projectability result known to the author was proved in \cite{mueller}. There, the above mentioned vector field $\mf Q$ was supposed to be tangential along the \emph{support surface} $ S$, which forces the corresponding stationary $\ml H$-surface to meet $S$ perpendicularly along its \emph{free trace} $\mf x|_I$. This condition was essential at many points of the proof in \cite{mueller}, in particular, while deriving the second variation formula for $A_{\mf Q}$ and establishing a boundary condition for the third component of the surface normal of our $\ml H$-surface.
One motivation for writing the present paper was to drop this restriction and to study $\ml H$-surfaces which meet $ S$ non-perpendicularly.

Methodically, we orientate on \cite{mueller} which in turn is based on the work of Hildebrandt and Sauvigny in \cite{hildefritz} and on Sauvigny's paper \cite{fritz}, where a corresponding projectability result for stable $\ml H$-surfaces subject to Plateau type boundary conditions has been proven. 

The paper is organized as follows:
In Section\,2 we fix notations, specify our assumptions and state the main projectability result, Theorem\,\ref{t1},  as well as some preliminary results on the $\ml H$-surface and its normal. The consequential unique solvability of the studied partially free problem is captured in Corollary\,\ref{c1}. In Section\,3 we derive the second variation formula for the functional $A_{\mf Q}$ allowing boundary perturbations on the free trace $\mf x|_I$. Then, Section\,4 contains the crucial boundary condition for the third component of the surface normal and the proof of Theorem\,\ref{t1}. We close with an exemplary application of Theorem\,\ref{t1} to the existence question for a mixed boundary value problem for the non-parametric $\ml H$-surface equation, Corollary\,\ref{c2}.

%%%%%%%%%%%%%%%%%%%%%%%%%%%%%%%%%%%%%%%%%%%%%%%%%%%%%%%%%%%%%%%%%%%%%%%%%%%%%%%%%%%%%%%%%%%%%%%%%%%%%%%%%%%%

\setcounter{equation}{0}
\section{Notations and main result}
We start by specifying our additional assumptions on the boundary configuration: Let $\{\Gamma, S\}$ be a projectable boundary configuration in the sense of Definition\,\ref{d1}. Let $\sigma=\sigma(s)$, $s\in[0,s_0]$, parametrize $\Sigma=\pi(S)$ by arc length, that is, 
	$$\sigma\in C^3([0,s_0],\mb R^2),\quad|\sigma'|\equiv1\ \mbox{on}\ [0,s_0],\quad\mbox{and}\quad s_0=\mbox{length}(\Sigma)
	>0.$$
Setting $\mf e_3:=(0,0,1)$ we define $C^2$-unit tangent and normal vector fields $\mf t,\mf n$ on $ S$ as follows:
	\bee\label{g0.1}
	\mf t(\mf p):=(\sigma'(s),0),\quad\mf n(\mf p):=\mf t(\mf p)\wedge\mf e_3\quad\mbox{for}\ \mf p\in\{\sigma(s)\}\wedge\mb R,\ s\in[0,s_0].
	\ee
Furthermore, we can write $\Gamma=\{(x^1,x^2,\gamma(x^1,x^2))\in\mb R^3\,:\ (x^1,x^2)\in\underline\Gamma\}$, where $\underline\Gamma=\pi(\Gamma)$ is a closed $C^3$-Jordan arc and $\gamma\in C^3(\underline\Gamma)$ is the \emph{height function}. For the end points $\mf p_1,\mf p_2$ of $\Gamma$ we assume to have representations 
	$$\mf p_1=\big(\sigma(0),\gamma(\sigma(0)\big),\quad \mf p_2=\big(\sigma(s_0),\gamma(\sigma(s_0)\big).$$
The set $\underline\Gamma\cup\Sigma$ bounds a simply connected domain $G\subset\mb R^2$, that is, $\partial G=\underline\Gamma\cup\Sigma$, and we have $\underline\Gamma\cap\Sigma=\{\pi_1,\pi_2\}$ with $\pi_j=\pi(\mf p_j)$, $j=1,2$. With $\alpha_j\in (0,\pi)$ we denote the interior angle between $\underline\Gamma$ and $\Sigma$ at $\pi_j$ w.r.t.~$G$ ($j=1,2$). Finally, we assume that $\Sigma$ is parametrized such that $\nu:=\pi(\mf n)$ points to the exterior of $G$ along $\Sigma$.

\begin{definition}\label{d2}
A projectable boundary configuration $\{\Gamma, S\}$ is called \emph{$R$-admis\-sible}, if the following hold:
\begin{itemize}
\item[(i)]
$\Gamma\cup S\subset Z:=\{(p^1,p^2,p^3)\in\mb R^3\,:\ |(p^1,p^2)|<R\}$ for some $R>0$.
\item[(ii)]
$G$ is \emph{$\frac1R$--convex}, i.e., for any point $\xi\in\partial G$ there is an open disc $D_\xi\subset\mb R^2$ of radius $R$ such that $G\subset D_\xi$ and $\xi\in\partial D_\xi$.
\end{itemize}
\end{definition}

For a given $R$-admissible boundary configuration $\{\Gamma,S\}$, we define the class $\ml C(\Gamma, S;\overline{Z})$ of mappings $\mf x\in H^1_2(B^+,\overline{Z})$, which satisfy the boundary conditions (\ref{g0.2}) weakly, i.e.,
	\bee\label{g0.2+}
	\begin{array}{l}
	\mf x(w)\in S\quad\mbox{for a.a.~}\ w\in I,\\[0.5ex]
	\mf x|_{ J}: J\to\Gamma\ \mbox{continuously and weakly monotonic},\\[0.5ex]
	\mf x(-1)=\mf p_1,\ \mf x(+1)=\mf p_2.
	\end{array}
	\ee
For arbitrary $\mu\in[0,1)$, we additionally define its subsets
	\bee\label{g0.3}
	\ml C_\mu(\Gamma, S;\overline{Z}):=\left\{\mf x\in\ml C(\Gamma, S;\overline{Z})\,:\ 
	\begin{array}{l}
	\mf x\in C^\mu(\overline{B^+},\overline{Z}),\\[0.5ex]
	\mf x|_{J}:J\to\Gamma\ \mbox{strictly monotonic}
	\end{array}\right\}.
	\ee
Now let $\mf Q=\mf Q(\mf p)\in C^1(\overline{Z},\mb R^3)$ be a vector field satisfying
	\bee\label{g0.4}
	\begin{array}{l}
	\ds\sup_{p\in\overline{Z}}|\mf Q(\mf p)|<1,\\[2.5ex]
	\mbox{div}\,\mf Q(\mf p)=2\ml H(\mf p)\quad\mbox{for all}\ \mf p\in\overline{Z}.
	\end{array}
	\ee
Here the function $\ml H=\ml H(\mf p)$ belongs to $C^{1,\alpha}(\overline{Z})$ for some $\alpha\in(0,1)$ and fulfills
	\bee\label{g0.5}
	\sup_{\mf p\in\overline{Z}}|\ml H(\mf p)|\le\frac1{2R}.
	\ee
We introduce the functional
	\bee\label{g0.6}
	E_{\mf Q}(\mf x):=\iint\limits_{B^+}\Big\{\frac12|\nabla\mf x(w)|^2+\mf Q(\mf x)\cdot\mf x_u\wedge\mf x_v(w)\Big\}
	\,du\,dv,\quad\mf x\in H_2^1(B^+,\overline{Z}),
	\ee
and consider the variational problem
	\bee\label{g0.7}
	E_{\mf Q}(\mf x)\to\min,\quad\mf x\in\ml C(\Gamma, S;\overline{Z}).
	\ee
The following lemma collects some well known results concerning the existence and regularity of solutions of (\ref{g0.7}) as well as stationary points of $E_{\mf Q}$.

\begin{lemma}{\bf(Heinz, Hildebrandt, Tomi)}\label{l1}\\
Let $\{\Gamma,S\}$ be an $R$-admissible boundary configuration $\{\Gamma,S\}$ and assume $\mf Q\in C^1(\overline{Z},\mb R^3)$, $H\in C^{1,\alpha}(\overline{Z})$ to satisfy (\ref{g0.4}) and (\ref{g0.5}). Then there exists a solution $\mf x=\mf x(w)$ of (\ref{g0.7}). $\mf x$ belongs to the class $\ml C_\mu(\Gamma, S;\overline{Z})\cap C^{3,\alpha}(B^+,Z)$ for some $\mu\in(0,1)$ and satisfies the system (\ref{g0.8}), i.e., $\mf x$ is a partially free $\ml H$-surface.

More generally, any stationary point $\mf x\in \ml C_0(\Gamma, S;\overline{Z})$ of $E_{\mf Q}$ solves (\ref{g0.8}) and belongs to the class $C^{3,\alpha}(B^+,Z)$. Here, stationarity means
	$$\lim_{\varepsilon\to0+}\frac1\varepsilon\big\{E_{\mf Q}(\mf x_\varepsilon)-E_{\mf Q}(\mf x)\big\}\ge0$$
for all inner and outer variations $\mf x_\varepsilon\in{\cal C}_0(\Gamma,S;\overline{Z})$, $\varepsilon\in[0,\varepsilon_0)$ with sufficiently small $\varepsilon_0>0$; see Definition\,2 in \cite{dierkes} Section 5.4 for the definition of inner and outer variations. 
\end{lemma}

We also associate the \emph{generalized area functional} to $\mf Q$:
	\bee\label{g0.9}
	A_{\mf Q}(\mf x):=\iint\limits_{B^+}\Big\{|\mf x_u\wedge\mf x_v|+\mf Q(\mf x)\cdot\mf x_u\wedge\mf x_v(w)\Big\}
	\,du\,dv,\quad\mf x\in
	H_2^1(B^+,\overline{Z}).
	\ee
A stationary, partially free $\ml H$-surface $\mf x\in\ml C_0(\Gamma, S;\overline{Z})$ is called \emph{stable}, if it is stable w.r.t.~$A_{\mf Q}$, that means, the second variation $\frac{d^2}{d\varepsilon^2}A_{\mf Q}\big(\tilde{\mf x}(\cdot,\varepsilon)\big)|_{\varepsilon=0}$ of $A_{\mf Q}$ is nonnegative for all outer variations $\tilde{\mf x}(\cdot,\varepsilon)\in{\cal C}_0(\Gamma,S;\overline Z)$, $\varepsilon\in(-\varepsilon_0,\varepsilon_0)$, for which this quantity exists; note that $\mf x$ has its image $\mf x(\overline{B^+})$ in $Z$, according to Lemma\,\ref{l1}. Since the first variation of $A_{\mf Q}$ w.r.t.~such variations $\tilde{\mf x}$ vanishes for stationary $\mf x$, any relative minimizer of $A_\mf Q$ in ${\cal C}_0(\Gamma,S;\overline{Z})$ is stable. In Definition\,\ref{d5} below, we give an exact definition of stability, which is used in the present paper and which is somewhat less stringent than the above mentioned requirement. 

We are now in a position to state our main result:

\begin{theorem}\label{t1}
Let $\{\Gamma, S\}$ be an admissible boundary configuration and let $\mf Q\in C^{1,\alpha}(\overline{Z},\mb R^3)$ be chosen such that (\ref{g0.4}) is fullfilled with some $\ml H\in C^{1,\alpha}(\overline{Z})$, $\alpha\in(0,1)$, satisfying (\ref{g0.5}). In addition, we assume 
	\bee\label{g0.10}
	\frac\partial{\partial p^3}\ml H(\mf p)\ge0\quad\mbox{for all}\ \mf p\in\overline{Z}
	\ee
as well as
	\bee\label{g0.11}
	\begin{array}{l}
	(\mf Q\cdot\mf n)(\mf p)= (\mf Q\cdot\mf n)(p^1,p^2,0)\quad\mbox{for all}\ \mf p=(p^1,p^2,p^3)\in{ S},\\[2ex]
	|(\mf Q\cdot\mf n)(\mf p_j)|<\cos\alpha_j,\quad j=1,2.
	\end{array}
	\ee
Then any stable $\ml H$-surface $\mf x\in \ml C_\mu(\Gamma, S;\overline{Z})$, $\mu\in(0,1)$, possesses a graph representation over $\overline G$. More precisely, $\mf x$ is immersed and can be represented as the graph of some function $\zeta:\overline G\to\mb R\in C^{3,\alpha}(G)\cap C^{2,\alpha}(\overline G\setminus\{\pi_1,\pi_2\})\cap C^0(\overline G)$, which satisfies the mixed boundary value problem
	\begin{eqnarray}
	\label{g0.12-1}&&\ds\mbox{div}\Big(\frac{\nabla\zeta}{\sqrt{1+|\nabla\zeta|^2}}\Big)=2\ml H(\cdot,\zeta)\quad\mbox{in}\ G,\\[1ex]
	\label{g0.12-2}&&\ds\frac{\nabla\zeta\cdot\nu}{\sqrt{1+|\nabla\zeta|^2}}=\psi\quad\mbox{on}\ \Sigma\setminus\{\pi_1,\pi_2\},\qquad
	\zeta=\gamma\quad\mbox{on}\ \underline\Gamma.
	\end{eqnarray}
Here $\nu=\pi(\mf n)$ denotes the exterior unit normal on $\Sigma$ w.r.t.~$G$ and we defined $\psi:=\mf Q\cdot\mf n|_{\Sigma}\in C^1(\Sigma)$.
\end{theorem}

As a consequence of Theorem\,\ref{t1} we obtain the following

\begin{corollar}\label{c1}
Let the assumptions of Theorem\,\ref{t1} be satisfied. Then, apart from reparametrization, there exists exactly one stable $\ml H$-surface $\mf x\in \ml C_\mu(\Gamma, S;\overline{Z})$ with some $\mu\in(0,1)$.
\end{corollar}

\begin{proof}
The existence of a stable $\ml H$-surface $\mf x\in\ml C_\mu(\Gamma,S;\overline{Z})$ for some $\mu\in(0,1)$ - namely the solution of (\ref{g0.7}) - is assured by Lemma\,\ref{l1}. According to Theorem\,\ref{t1}, we can represent $\mf x$ as a graph over $G$, and the height function $\zeta$ solves the boundary value problem (\ref{g0.12-1}), (\ref{g0.12-2}).

If there would exist another stable $\ml H$-surface $\tilde{\mf x}\in\ml C_{\tilde\mu}(\Gamma,S;\overline{Z})$ with some $\tilde\mu\in(0,1)$ and if $\tilde\zeta$ denotes the height function of its graph representation, which also solves (\ref{g0.12-1}), (\ref{g0.12-2}) by Theorem\,\ref{t1}, we consider the difference function $g:=\zeta-\tilde\zeta$. As is well known, $g$ solves a linear elliptic differential equation in $G$, which is subject to the maximum principle according to assumption (\ref{g0.10}); cf.~\cite{sauvigny1} Chap.~VI, \S\,2. Consequently, $g$ assumes its maximum and minimum on $\partial G=\Sigma\cup\underline\Gamma$.

Assume that $g$ has a positive maximum at $p_0\in\Sigma\setminus\{\pi_1,\pi_2\}$. Then Hopf's boundary point lemma implies 
	$$\nabla g(p_0)=(\nabla g(p_0)\cdot\nu(p_0))\nu(p_0)\quad\mbox{with}\quad\nabla g(p_0)\cdot\nu(p_0)>0.$$ 
On the other hand, the first boundary condition in (\ref{g0.12-2}) yields $(M(p_0)\nabla g(p_0))\cdot\nu(p_0)=0$, where we abbreviated
	$$M(p):=\int\limits_0^1 Dh\big(t\nabla\zeta(p)+(1-t)\nabla\tilde\zeta(p)\big)\,dt,\quad p\in\Sigma,$$
with $h(z):=\frac z{\sqrt{1+|z|^2}}$, $z\in\mb R^2$. If we further note
	$$\big(Dh(z)\xi\big)\cdot\xi=\frac{|\xi|^2(1+|z|^2)-(\xi\cdot z)^2}{(1+|z|^2)^{\frac 32}}>0,\quad\xi\in\mb R^2\setminus\{0\},\quad z\in\mb R^2,$$
we deduce that $M$ is positive definite on $\Sigma$ and arrive at the contradiction
	$$0=(M(p_0)\nabla g(p_0))\cdot\nu(p_0)=(\nabla g(p_0)\cdot\nu(p_0))(M(p_0)\nu(p_0))\cdot\nu(p_0)>0.$$
Hence, we conclude $g\le0$ on $\overline G$ and, similarly, one proves $g\ge0$ on $\overline G$. This gives $\zeta\equiv\tilde\zeta$ on $\overline G$, which yields $\mf x=\tilde{\mf x}\circ\omega$ with some positively oriented para\-meter transformation $\omega:\overline{B^+}\to\overline{B^+}$. This proves the corollary.
\end{proof}
 
We complete this section with a preparatory lemma, which collects some analytical and geometrical regularity results and preliminary informations towards the projectability of our $\ml H$-surfaces:

\goodbreak
\begin{lemma}\label{l2}
Let the assumptions of Theorem\,\ref{t1} be satisfied and let $\mf x=\mf x(w)\in\ml C_\mu(\Gamma, S;\overline{Z})$ be an $\ml H$-surface which is stationary w.r.t.~$E_{\mf Q}$. Then there follow:
\begin{itemize}
\item[(i)]
 $\mf x\in C^{3,\alpha}(B^+,Z)\cap C^{2,\alpha}(\overline{B^+}\setminus\{-1,+1\},Z)$, and there holds
	\bee\label{g0.13}
	(\mf x_v+\mf Q(\mf x)\wedge\mf x_u)(w)\perp T_{\mf x(w)}S\quad\mbox{for all}\ w\in I,
	\ee
where $T_{\mf p}S$ denotes the tangential plane of $S$ at the point $\mf p\in S$.
\item[(ii)]
$f(\overline{B^+})\subset\overline G$ for the projected mapping $f:=\pi(\mf x)$.
\item[(iii)]
$\nabla\mf x(w)\not=\bs0$ for all $w\in\partial B^+\setminus\{-1,+1\}$, and $\nabla\mf x=\bs0$ for at most finitely many points in $B^+$.
\item[(iv)]
Set $W:=|\mf x_u\wedge\mf x_v|$, $B':=\{w\in B^+\,:\ W(w)>0\}$, and define the surface normal $\mf N(w):=W^{-1}\mf x_u\wedge\mf x_v(w)$ as well as the Gaussian curvature $K=K(w)$ of $\mf x$ for points $w\in B'$. Then $\mf N$ and $KW$ can be extended to mappings 
	\begin{eqnarray*}
	&&\mf N\in C^{2,\alpha}(B^+,\mb R^3)\cap C^{1,\alpha}(\overline{B^+}\setminus\{-1,+1\},\mb R^3)\cap C^0(\overline{B^+},\mb R^3),\\[1ex]
	&& KW\in C^{1,\alpha}(B^+),
	\end{eqnarray*}
and $\mf N$ satisfies the differential equation
	\bee\label{g0.13+}
	\Delta\mf N+2\big(2\ml H(\mf x)^2-K-(\nabla\ml H(\mf x)\cdot\mf N)\big)W\mf N=-2W\nabla\ml H(\mf x)\quad\mbox{in}\ B^+.
	\ee
\end{itemize}
\end{lemma}

\begin{proof}

\begin{description}
\item[\it (ii)]
Due to Lemma\,\ref{l1}, $\mf x$ is a stationary, partially free $\ml H$-surface of class $C^{3,\alpha}(B^+,Z)$. In addition, we have $f(\partial B^+)=\partial G$ due to the geometry of our boundary configuration. An inspection of the proof of Hilfssatz 4 of \cite{fritz} shows, that this boundary condition, the smallness condition (\ref{g0.5}) and the $\frac1R$--convexity of $G$ imply $f(\overline{B^+})\subset\overline G$ according to the maximum principle. 

\item[\it (i), (iii)]
A well known regularity result due to E.\,Heinz \cite{heinz} implies that $\mf x\in C^{2,\alpha}(B^+\cup J,Z)$. And from Theorem\,1 in \cite{mueller5} we obtain $\mf x\in C^{1,\frac12}(B^+\cup I,Z)$. Setting
	$$I':=\big\{w\in I\,:\ f(w)=(\pi\circ\mf x)(w)\not\in\{\pi_1,\pi_2\}\big\},$$ 
the stationarity yields the natural boundary condition (\ref{g0.13}) on $I'$.

Due to (ii), the arguments from Satz 2 in \cite{fritz} yield $\nabla\mf x(w)\not=\bs0$ for all $w\in J$. Assume that $w_0\in I$ is a branch point of $\mf x$ and set $B^+_\delta(w_0):=\{w\in B^+\,:\ |w-w_0|<\delta\}$. Then the asymptotic expansion from Theorem\,2 in \cite{mueller5} imply that $\mf x|_{B^+_\delta(w_0)}$, $0<\delta\ll 1$, looks like a whole perturbed disc. Consequently, the projection $f|_{B^+_\delta(w_0)}$ would meet the complement of $\overline G$, in contrast to $f(\overline B)\subset\overline G$. Indeed, for $w_0\in I'$ this effects from the natural boundary condition (\ref{g0.13}), which can be rewritten as $(\mf Q\cdot\mf n)(\mf x)=-\mf N\cdot\mf n(\mf x)$ on $I'$; see Remark\,\ref{r1} below. And for $w_0\in I\setminus I'$, i.e.~$f(w_0)\in\{\pi_1,\pi_2\}$, this is trivial by geometry. Consequently, we have a contradiction and $\nabla\mf x(w)\not=\bs0$ for $w\in I$ follows; this completes the proof of the first part of (iii). 

Next we show $I'=I$, i.e.~$f(I)=\Sigma\setminus\{\pi_1,\pi_2\}$. From \cite{hildebrandt-jager} or \cite{mueller3} we then obtain $\mf x\in C^{2,\alpha}(B^+\cup I,Z)$ and (\ref{g0.13}) holds on $I$; this will complete the proof of (i).

Assume there exists $w^*\in I$ with $f(w^*)=\pi_1$. Then there would be a maximal point $w_0\in I$ with $f(w_0)=\pi_1$ and $f(w)\in\Sigma\setminus\{\pi_1,\pi_2\}$ for $w\in(w_0,w_0+\varepsilon)\subset I$, $0<\varepsilon\ll 1$. Consequently, the boundary condition (\ref{g0.13}) holds on $(w_0,w_0+\varepsilon)$ and, in particular, we get
	\bee\label{g0.13++}
	(\mf x_v+\mf Q(\mf x)\wedge\mf x_u)\cdot\mf t(\mf x)=0\quad\mbox{on}\ (w_0,w_0+\varepsilon).
	\ee
By continuity, (\ref{g0.13++}) remains valid for $w=w_0$. In addition, the geometry of $S$ yields $\mf x_u=\pm|\mf x_u|\mf e_3$ in $w_0$. This and the relation $\mf n=\mf t\wedge\mf e_3$ on $S$ imply
	\bee\label{g0.13+++}
	\mf x_v\cdot\mf t(\mf x)=\pm|\mf x_u|\,\mf Q(\mf x)\cdot\mf n(\mf x)\quad\mbox{in}\ w_0.
	\ee
According to the conformality relations and $\nabla\mf x\not=\bs0$ on $I$, we have $|\mf x_u|=|\mf x_v|\not=0$ in $w_0$. Denote the angle between $\mf x_v(w_0)$ and $\mf t(\mf x(w_0))$ by $\beta_1$. Then (\ref{g0.13+++}) and condition (\ref{g0.11}) imply
	$$|\cos\beta_1|=|\mf Q(\mf x(w_0))\cdot\mf n(\mf x(w_0))|<\cos\alpha_1\quad\mbox{or}\quad\beta_1\in(\alpha_1,\pi-\alpha_1),$$
where $\alpha_1\in(0,\frac\pi2)$ denotes the interior angle between $\ul\Gamma$ and $\Sigma$ at $\pi_1$ w.r.t.~$G$. A simple application of the mean-value theorem then yields a contradiction to the inclusion $f(\overline{B^+})\subset\overline G$. Analogously, one shows that there cannot exist $w^{**}\in I$ with $f(w^{**})=\pi_2$. In conclusion, we have $I'=I$ and (i) is proved.

We finally show the finiteness of branch points in $B^+$, completing the proof of (iii): Hildebrandt's asymptotic expansions at interior branch points \cite{hildebrandt} imply the isolated character of these points. By $\nabla\mf x\not=\bs0$ on $I\cup J$, the only points where branch points could accumulate are the corner points $w=\pm1$. But this is impossible, too, according to the asymptotic expansions near these points proven in \cite{mueller4} Theorem 2.2; see Corollary\,7.1 there. We emphasize that the cited result is applicable, since $\Gamma$ and $S$ meet with positive angles $\gamma_j\in(0,\alpha_j]$ at $\mf p_j$ by Definition\,\ref{d1}, and since we assume 
	$$|\mf Q(\mf p_j)\cdot\mf n(\mf p_j)|<\cos\alpha_j\le\cos\gamma_j,\quad j=1,2.$$
(Note that a simple reflection of $S$ can be used to assure $\{\Gamma,S\}$ and $\mf x$ to fulfill the assumptions of \cite{mueller4} Corollary\,7.1.)

\item[\it (iv)]
The interior regularity $\mf N\in C^{2,\alpha}(B^+,\mb R^3)$, $KW\in C^{1,\alpha}(B^+)$ as well as equation (\ref{g0.13+}) were proven by F.\,Sau\-vig\-ny in \cite{fritz} Satz 1. The global regularity $\mf N\in C^{1,\alpha}(\overline{B^+}\setminus\{-1,+1\},\mb R^3)$ follows from (i) and (iii). Finally, the continuity of $\mf N$ up to the corner points $w=\pm1$ was proven in \cite{mueller4} Theorem\,5.4; see the remarks above concerning the applicability of this result. \vspace{-3ex}
\end{description}
\end{proof}

\begin{remark}\label{r1}
By taking the cross product with $\mf x_u\in T_{\mf x}S$, the natural boundary condition (\ref{g0.13}) can be written in the form
	\bee\label{g0.15}
	\mf Q(\mf x)\cdot\mf n(\mf x)=-\mf N\cdot\mf n(\mf x)\quad\mbox{on}\ I.
	\ee
This relation describes the well known fact that the normal component of $\mf Q$ w.r.t.~to $S$ prescribes the contact angle between a stationary $\ml H$-surface and the support surface $S$. Due to the smallness condition on $\mf Q$ in (\ref{g0.4}), it particularly follows that the $\ml H$-surface $\mf x$ meets $S$ non-tangentially along its free trace $\mf x|_I$.
\end{remark}
	
%%%%%%%%%%%%%%%%%%%%%%%%%%%%%%%%%%%%%%%%%%%%%%%%%%%%%%%%%%%%%%%%%%%%%%%%%%%%%%%%%%%%%%%%%%%%%%%%%%%%%%%%%%%%%%%%%%%

\setcounter{equation}{0}
\section{The second variation of $A_{\mf Q}$, stable $\ml H$-surfaces}
Let us choose an $\ml H$-surface $\mf x\in\ml C_\mu(\Gamma, S;\overline{Z})$, $\mu\in(0,1)$, which is stationary w.r.t.~$E_{\mf Q}$ (and thus belongs to $C^{3,\alpha}(B^+,Z)\cap C^{2,\alpha}(\overline{B^+}\setminus\{-1,+1\},Z)$ according to Lemma\,\ref{l2}\,(i)). Consider a one-parameter family $\tilde{\mf x}=\tilde{\mf x}(w,\varepsilon)$, which belongs to the class $C^\mu(\Gamma, S;\overline{Z})\cap C^2(\overline{B^+}\setminus\{-1,+1\},\mb R^3)$ for any fixed $\varepsilon\in(-\varepsilon_0,\varepsilon_0)$ and which depends smoothly on $\varepsilon$ together with its first and second derivatives w.r.t.~$u,v$. We call $\tilde{\mf x}$ an \emph{admissible perturbation} of $\mf x$, if we have:
\begin{itemize}
\item[(i)] $\tilde{\mf x}(w,0)=\mf x(w)$ for all $w\in\overline{B^+}$,
\item[(ii)] $\mbox{supp}\big(\tilde{\mf x}(\cdot,\varepsilon)-\mf x\big)\subset B^+\cup I$ for all $\varepsilon\in (-\varepsilon_0,\varepsilon_0)$,
\item[(iii)] $\mf y:=\frac\partial{\partial\varepsilon}\tilde{\mf x}(\cdot,\varepsilon)\big|_{\varepsilon=0}\in C^2_c(B^+\cup I,\mb R^3)$, $\mf z:=\frac{\partial^2}{\partial\varepsilon^2}\tilde{\mf x}(\cdot,\varepsilon)\big|_{\varepsilon=0}\in C^1_c(B^+\cup I,\mb R^3)$.
\end{itemize}
The \emph{direction} $\mf y=\frac\partial{\partial\varepsilon}\tilde{\mf x}(\cdot,\varepsilon)\big|_{\varepsilon=0}$ of an admissible perturbation $\tilde{\mf x}$ satisfies
	\bee\label{g1.1}
	\mf y(w)\in T_{\mf x(w)}S\quad\mbox{for all}\ w\in I.
	\ee
On the other hand, choosing an arbitrary vector-field $\mf y\in C^2_c(B^+\cup I,\mb R^3)$ with the property (\ref{g1.1}), one may construct an admissible perturbation $\tilde{\mf x}$ as described above by using a flow argument (compare, e.g., \cite{dierkes} pp.~32--33). 

In the present section, we compute the second variation $\frac{d^2}{d\varepsilon^2}A_{\mf Q}(\tilde{\mf x}(\cdot,\varepsilon))\big|_{\varepsilon=0}$ for admissible perturbations. To this end, we have to examine the quantity
	\bee\label{g1.4}
	\frac{\partial^2}{\partial\varepsilon^2}\big(|\tilde{\mf x}_u\wedge\tilde{\mf x}_v|+\mf Q(\tilde{\mf x})\cdot\tilde{\mf x}_u\wedge\tilde{\mf x}_v
	\big)\Big|_{\varepsilon=0}=\frac{\partial^2}{\partial\varepsilon^2}\big(|\tilde{\mf x}_u\wedge\tilde{\mf x}_v|\big)\Big|_{\varepsilon=0}
	+\frac{\partial^2}{\partial\varepsilon^2}\big(\mf Q(\tilde{\mf x})\cdot\tilde{\mf x}_u\wedge\tilde{\mf x}_v\big)\Big|_{\varepsilon=0}.
	\ee
We first compute (\ref{g1.4}) in the regular set $B'\cup I$ with
	$$B'=\{w\in B^+\,:\ W(w)>0\},\quad W=|\mf x_u\wedge\mf x_v|=|\mf x_u|^2=|\mf x_v|^2,$$
and then observe that the resulting formula can be extended continuously to $B^+\cup I$; note that $\mf x$ possesses no branch points on $I$ according to Lemma\,\ref{l2}\,(iii). 

We start with the first addend on the right-hand side of (\ref{g1.4}):

\begin{prop}\label{p1}
Let $\tilde{\mf x}$ be an admissible perturbation of a stationary $\ml H$-surface $\mf x\in \ml C_\mu(\Gamma,S,\overline Z)$ as described above. Define $\varphi:=\mf y\cdot\mf N\in C^2_c(B^+\cup I,\mb R^3)$. Then there holds
	\begin{eqnarray*}
	\frac{\partial^2}{\partial\varepsilon^2}\big(|\tilde{\mf x}_u\wedge\tilde{\mf x}_v|\big)\Big|_{\varepsilon=0}\!\!\!\!\!\!\!
	& =\!\!\!\! & |\nabla\varphi|^2+2KW\varphi^2-2\ml H(\mf x)\mf y\cdot(\mf y_u\wedge\mf x_v+\mf x_u\wedge\mf y_v)\big]\\
	&& +2\ml H(\mf x)\!\big[\varphi(\mf x_u\cdot\mf y_u)+\varphi(\mf x_v\cdot\mf y_v)+(\mf x_u\cdot\mf y)
	\varphi_u+(\mf x_v\cdot\mf y)\varphi_v\!\big]\\[1ex]
	&& -\big[\varphi\big(\mf N_u+2\ml H(\mf x)\mf x_u\big)\cdot\mf y\big]_u-\big[\varphi\big(\mf N_v+2\ml H(\mf x)\mf x_v\big)\cdot\mf y\big]_v\\[1ex]
	&& +\big[\mf N\cdot(\mf y\wedge\mf y_v)\big]_u+\big[\mf N\cdot(\mf y_u\wedge\mf y)\big]_v\\[0.5ex]
	&& -2\ml H(\mf x)\mf z\cdot(\mf x_u\wedge\mf x_v)+(\mf z\cdot\mf x_u)_u+(\mf z\cdot\mf x_v)_v\quad\mbox{on}\ B',
	\end{eqnarray*}
where $K$ denotes the Gaussian curvature of $\,\mf x$. 
\end{prop}

\begin{proof}
\begin{enumerate}
\item
We start by noting the relation
	\bee\label{g1.5}
	\frac{\partial^2}{\partial\varepsilon^2}\big(|\tilde{\mf x}_u\wedge\tilde{\mf x}_v|\big)\Big|_{\varepsilon=0}
	\!=\frac1{2W}\frac{\partial^2}{\partial\varepsilon^2}\big(|\tilde{\mf x}_u\wedge\tilde{\mf x}_v|^2\big)\Big|_{\varepsilon=0}\!
	-\frac1{4W^3}\Big[\frac{\partial}{\partial\varepsilon}\big(|\tilde{\mf x}_u\wedge\tilde{\mf x}_v|^2\big)\Big]^2\Big|_{\varepsilon=0}\!\!
	\quad
	\ee
on $B'$. Expanding $\tilde{\mf x}$ w.r.t.~$\varepsilon$, we infer
	\bee\label{g1.5+}
	\tilde{\mf x}(\cdot,\varepsilon)=\mf x+\varepsilon\mf y+\frac{\varepsilon^2}2\mf z+o(\varepsilon^2)\quad\mbox{on}\ B^+
	\ee
and, consequently, 
	\bee\label{g1.6}
	\begin{array}{rcl}
	\tilde{\mf x}_u\wedge\tilde{\mf x}_v &=& W\mf N+\varepsilon(\mf x_u\wedge\mf y_v+\mf y_u\wedge\mf x_v)+\varepsilon^2\mf y_u\wedge\mf y_v\\[1ex]
	&& \ds +\frac{\varepsilon^2}2(\mf x_u\wedge\mf z_v+\mf z_u\wedge\mf x_v)+o(\varepsilon^2)\quad\mbox{on}\ B'
	\end{array}
	\ee
as well as 
	\bee\label{g1.7}
	\begin{array}{rcl}
	|\tilde{\mf x}_u\wedge\tilde{\mf x}_v|^2 &=& W^2+2\varepsilon W\mf N\cdot(\mf x_u\wedge\mf y_v+\mf y_u\wedge\mf x_v)\\[1ex]
	&& +\varepsilon^2|\mf x_u\wedge\mf y_v+\mf y_u\wedge\mf x_v|^2+2\varepsilon^2W\mf N\cdot(\mf y_u\wedge\mf y_v)\\[1ex]
	&& +\varepsilon^2W\mf N\cdot(\mf x_u\wedge\mf z_v+\mf z_u\wedge\mf x_v)+o(\varepsilon^2).
	\end{array}
	\ee
Combining (\ref{g1.5}) with (\ref{g1.7}) gives
	$$\begin{array}{rcl}
	\ds\frac{\partial^2}{\partial\varepsilon^2}\big(|\tilde{\mf x}_u\wedge\tilde{\mf x}_v|\big)\Big|_{\varepsilon=0} 
	& = & W^{-1}|\mf x_u\wedge\mf y_v+\mf y_u\wedge\mf x_v|^2+2\mf N\cdot(\mf y_u\wedge\mf y_v)\\[1ex]
	&& +\mf N\cdot(\mf x_u\wedge\mf z_v+\mf z_u\wedge\mf x_v)\\[1ex]
	&& -W^{-1}\big[\mf N\cdot(\mf x_u\wedge\mf y_v+\mf y_u\wedge\mf x_v)\big]^2\\[1ex]
	& = & (\mf y_u\cdot\mf N)^2+(\mf y_v\cdot\mf N)^2+2\mf N\cdot(\mf y_u\wedge\mf y_v)\\[1ex]
	&& +\mf N\cdot(\mf x_u\wedge\mf z_v+\mf z_u\wedge\mf x_v).
	\end{array}$$
And since $\mf x$ is a conformally parametrized $\ml H$-surface, we have
	\begin{eqnarray*}
	\mf N\cdot(\mf x_u\wedge\mf z_v+\mf z_u\wedge\mf x_v)\! & =\! & \mf z_v\cdot\mf x_v+\mf z_u\cdot\mf x_u\\
	& =\! & (\mf z\cdot\mf x_u)_u+(\mf z\cdot\mf x_v)_v-2\ml H(\mf x)W\mf z\cdot\mf N\quad\mbox{on}\ B',
	\end{eqnarray*}
arriving at
	\bee\label{g1.8}
	\begin{array}{rcl}
	\ds\frac{\partial^2}{\partial\varepsilon^2}\big(|\tilde{\mf x}_u\wedge\tilde{\mf x}_v|\big)\Big|_{\varepsilon=0} \!
	& = & \!(\mf y_u\cdot\mf N)^2+(\mf y_v\cdot\mf N)^2+2\mf N\cdot(\mf y_u\wedge\mf y_v)\\[1ex]
	&& \!+(\mf z\cdot\mf x_u)_u+(\mf z\cdot\mf x_v)_v-2\ml H(\mf x)W\mf z\cdot\mf N\quad\mbox{on}\ B'.
	\end{array}
	\ee
\item
In the following, we sometimes write $u^1:=u$, $u^2:=v$ and use Einstein's convention summing up tacitly over sub- and superscript latin indizes from $1$ to $2$. Furthermore, we set $\lambda^j:=W^{-1}\mf x_{u^j}\cdot\mf y$ for $j=1,2$ obtaining
	$$\mf y=\lambda^j\mf x_{u^j}+\varphi\mf N\quad\mbox{on}\ B'.$$
Writing $g_{jk}:=\mf x_{u^j}\cdot\mf x_{u^k}$, $g^{jk}$, $\Gamma_{jk}^l$, and $h_{jk}:=\mf x_{u^ju^k}\cdot\mf N=-\mf x_{u^j}\cdot\mf N_{u^k}$ for the coefficients of the first fundamental form, its inverse and Christoffel symbols, and the coefficients of the second fundamental form, respectively, we then infer	
	\bee\label{g1.9}
	\mf y_{u^k}=\big(\lambda^j_{u^k}+\lambda^l\Gamma_{lk}^j-\varphi h_{kl}g^{lj}\big)\mf x_{u^j}+\big(\lambda^j h_{jk}+\varphi_{u^k}\big)\mf N\quad
	\mbox{on}\ B'.
	\ee
Due to the conformal parametrization of the $\ml H$-surface $\mf x$, we have
	\bee\label{g1.10}
	\begin{array}{l}
	\ds g_{jk}=W\delta_{jk},\quad g^{jk}=\frac{\delta^{jk}}{W},\\[2ex]
	\ds\Gamma_{11}^1=-\Gamma_{22}^1=\Gamma_{12}^2=\Gamma_{21}^2=\frac{W_u}{2W},\\[2ex]
	\ds\Gamma_{22}^2=-\Gamma_{11}^2=\Gamma_{21}^1=\Gamma_{12}^1=\frac{W_v}{2W},\\[3ex]
	h_{11}+h_{22}=2W\ml H(\mf x),\quad h_{11}h_{22}-(h_{12})^2=W^2K\quad\mbox{on}\ B',
	\end{array}
	\ee
where $\delta_{jk}=\delta^{jk}$ denotes the Kronecker delta.
\item
We now evaluate the first line of the right-hand side in (\ref{g1.8}): Using (\ref{g1.9}) and (\ref{g1.10}), the first two terms can be written as
	\bee\label{g1.11}
	\begin{array}{l}
	\!\!(\mf y_u\cdot\mf N)^2+(\mf y_v\cdot\mf N)^2 = (\lambda^1h_{11}+\lambda^2h_{12}+\varphi_u)^2+(\lambda^1h_{12}+\lambda^2h_{22}+\varphi_v)^2
	\\[1.5ex]
	\hspace*{10ex}= |\nabla\varphi|^2+\big[(\lambda^1)^2+(\lambda^2)^2\big](h_{12})^2+(\lambda^1)^2(h_{11})^2+(\lambda^2)^2(h_{22})^2\\[1.5ex]
	\hspace*{12.5ex} + 4\lambda^1\lambda^2h_{12}W\ml H(\mf x)+4(\lambda^1\varphi_u+\lambda^2\varphi_v)W\ml H(\mf x)\\[1.5ex]
	\hspace*{12.5ex} + 2(\lambda^2h_{12}-\lambda^1h_{22})\varphi_u+2(\lambda^1h_{12}-\lambda^2h_{11})\varphi_v\quad\mbox{on}\ B'.
	\end{array}
	\ee
We next write the third term on the right-hand side of (\ref{g1.8}) as
	\bee\label{g1.12}
	\begin{array}{rcl}
	2\mf N\cdot(\mf y_u\wedge\mf y_v) & = & [\mf N\cdot(\mf y\wedge\mf y_v)]_u+[\mf N\cdot(\mf y_u\wedge\mf y)]_v\\[1.5ex]
	&& -\mf N_u\cdot(\mf y\wedge\mf y_v)-\mf N_v\cdot(\mf y_u\wedge\mf y)\quad\mbox{on}\ B'.
	\end{array}
	\ee
Using the relations $\mf N\wedge\mf x_u=\mf x_v$, $\mf N\wedge\mf x_v=-\mf x_u$, we get from (\ref{g1.9}):
	\begin{eqnarray*}
	\mf y\wedge\mf y_{u^k} & = & -\varphi\big(\lambda^2_{u^k}+\lambda^1\Gamma_{1k}^2+\lambda^2\Gamma_{2k}^2-\varphi h_{k2}W^{-1}\big)\mf x_u\\[1ex] 
	&& +\varphi\big(\lambda_{u^k}^1+\lambda^1\Gamma_{1k}^1+\lambda^2\Gamma_{2k}^1-\varphi h_{k1}W^{-1}\big)\mf x_v\\[1ex]
	&& +\lambda^2\big(\lambda^1h_{1k}+\lambda^2h_{2k}+\varphi_{u^k}\big)\mf x_u\\[1ex]
	&& -\lambda^1\big(\lambda^1h_{1k}+\lambda^2h_{2k}+\varphi_{u^k}\big)\mf x_v+(\ldots)\mf N\quad\mbox{on}\ B',
	\end{eqnarray*}
where $(\ldots)\mf N$ denotes the normal part of $\mf y\wedge\mf y_{u^k}$. This identity, formula (\ref{g1.10}), and the Weingarten equations $\mf N_{u^j}=-h_{jk}g^{kl}\mf x_{u^l}$ on $B'$ yield
	\bee\label{g1.13}
	\begin{array}{rcl}
	&& \ds\hspace*{-6ex}-\mf N_u\cdot(\mf y\wedge\mf y_v)-\mf N_v\cdot(\mf y_u\wedge\mf y)\\[1.5ex]
	&& \ds\hspace*{-3.3ex}= W^{-1}\big[(h_{11}\mf x_u+h_{12}\mf x_v)\cdot(\mf y\wedge \mf y_v)-(h_{21}\mf x_u+h_{22}\mf x_v)\cdot(\mf y\wedge\mf y_u)
	\big]\\[1.5ex]
	&& \ds\hspace*{-3.3ex}= 2(\varphi)^2WK+(\lambda^1h_{22}-\lambda^2h_{12})\varphi_u-(\lambda^1h_{12}-\lambda^2h_{11})\varphi_v\\[1.5ex]
	&& \ds\hspace*{-0.8ex} +\varphi\big[\lambda_v^1h_{12}-\lambda^1_uh_{22}-\lambda^1W_u\ml H(\mf x)\big]
	-\varphi\big[\lambda_v^2h_{11}-\lambda^2_uh_{12}+\lambda^2W_v\ml H(\mf x)\big]\\[1.5ex]
	&& \ds\hspace*{-0.8ex} +\big[(\lambda^1)^2+(\lambda^2)^2\big]\big[h_{11}h_{22}-(h_{12})^2\big]\quad\mbox{on}\ B'.
	\end{array}
	\ee
According to the Codazzi-Mainardi equations
	$$h_{21,v}-h_{22,u}+W_uH=0,\quad h_{11,v}-h_{12,u}-W_vH=0,$$
we infer
	\begin{eqnarray*}
	&& \lambda_v^1h_{12}-\lambda^1_uh_{22}-\lambda^1W_u\ml H(\mf x)=(\lambda^1h_{12})_v-(\lambda^1h_{22})_u,\\[1ex]
	&& \lambda_v^2h_{11}-\lambda^2_uh_{12}+\lambda^2W_v\ml H(\mf x)=(\lambda^2h_{11})_v-(\lambda^2h_{12})_u\quad\mbox{on}\ B'.
	\end{eqnarray*}
Inserting these identities into (\ref{g1.13}) and the resulting relation into (\ref{g1.12}), we arrive at
	\bee\label{g1.14}
	\begin{array}{l}
	\hspace*{-0.8ex}2\mf N\cdot(\mf y_u\wedge\mf y_v) = [\mf N\cdot(\mf y\wedge\mf y_v)]_u+[\mf N\cdot(\mf y_u\wedge\mf y)]_v\\[1.5ex]
	\hspace*{16.8ex}+2(\varphi)^2WK+(\lambda^1h_{22}-\lambda^2h_{12})\varphi_u-(\lambda^1h_{12}-\lambda^2h_{11})\varphi_v\\[1.5ex]
	\hspace*{16.8ex}-\varphi(\lambda^1h_{22}-\lambda^2h_{12})_u+\varphi(\lambda^1h_{12}-\lambda^2h_{11})_v\\[1.5ex]
	\hspace*{16.8ex}+\big[(\lambda^1)^2+(\lambda^2)^2\big]\big[h_{11}h_{22}-(h_{12})^2\big]\quad\mbox{on}\ B'.
	\end{array}
	\ee
Adding (\ref{g1.11}) and (\ref{g1.14}) we now find
	\bee\label{g1.15}
	\begin{array}{l}
	\hspace*{-1ex}(\mf y_u\cdot\mf N)^2+(\mf y_v\cdot\mf N)^2+2\mf N\cdot(\mf y_u\wedge\mf y_v)\\[1.5ex]
	\hspace*{5ex} = |\nabla\varphi|^2+2(\varphi)^2KW+[\mf N\cdot(\mf y\wedge\mf y_v)]_u+[\mf N\cdot(\mf y_u\wedge\mf y)]_v\\[1.5ex]
	\hspace*{7.5ex} -\big[\varphi(\lambda^1h_{22}-\lambda^2h_{12})\big]_u+\big[\varphi(\lambda^1h_{12}-\lambda^2h_{11})\big]_v\\[1.5ex]
	\hspace*{7.5ex} +2W\ml H(\mf x)\big[(\lambda^1)^2h_{11}+(\lambda^2)^2h_{22}+2\lambda^1\lambda^2h_{12}+2(\lambda^1\varphi_u+\lambda^2\varphi_v)\big]
	\end{array}
	\ee
on $B'$. Finally, we calculate via the Weingarten equations and (\ref{g1.10})
	\bee\label{g1.16}
	\begin{array}{rcl}
	\lambda^1h_{22}-\lambda^2h_{12} & = & W^{-1}(h_{22}\mf x_u-h_{12}\mf x_v)\cdot\mf y=(\mf N_u+2\ml H(\mf x)\mf x_u)\cdot\mf y,\\[1.5ex]
	\lambda^1h_{12}-\lambda^2h_{11} & = & W^{-1}(h_{12}\mf x_u-h_{11}\mf x_v)\cdot\mf y=-(\mf N_v+2\ml H(\mf x)\mf x_v)\cdot\mf y
	\end{array}
	\ee
as well as
	\bee\label{g1.17}
	\begin{array}{l}
	(\lambda^1)^2h_{11}+(\lambda^2)^2h_{22}+2\lambda^1\lambda^2h_{12}+2(\lambda^1\varphi_u+\lambda^2\varphi_v)\\[1.5ex]
	\hspace*{8ex} = \lambda^1(\lambda^1h_{11}+\lambda^2h_{12})+\lambda^2(\lambda^1h_{12}+\lambda^2h_{22})+2(\lambda^1\varphi_u+\lambda^2\varphi_v)
	\\[1.5ex]
	\hspace*{8ex} = -\lambda^1(\mf N_u\cdot\mf y)-\lambda^2(\mf N_v\cdot\mf y)+2(\lambda^1\varphi_u+\lambda^2\varphi_v)\\[1.5ex]
	\hspace*{8ex} = W^{-1}\big[(\mf x_u\cdot\mf y)(\mf N\cdot\mf y_u)+(\mf x_v\cdot\mf y)(\mf N\cdot\mf y_v)\big]
	+(\lambda^1\varphi_u+\lambda^2\varphi_v)\\[1.5ex]
	\hspace*{8ex} = W^{-1}\big[\varphi(\mf x_u\cdot\mf y_u)+\varphi(\mf x_v\cdot\mf x_v)+(\mf x_u\cdot\mf y)\varphi_u+(\mf x_v\cdot\mf y)\varphi_v\big]
	\\[1.5ex]
	\hspace*{10.5ex} - W^{-1}\big[\mf y\cdot(\mf y_u\wedge\mf x_v)+\mf y\cdot(\mf x_u\wedge\mf y_v)\big].
	\end{array}
	\ee
Inserting (\ref{g1.16}) and (\ref{g1.17}) into (\ref{g1.15}), the asserted identity follows from the resulting relation and formula (\ref{g1.8}).\vspace*{-2.5ex}
\end{enumerate}
\end{proof}

\begin{prop}\label{p2}
Under the assumptions of Proposition\,\ref{p1}, there holds
	\begin{eqnarray*}
	&& \hspace*{-4ex}\frac{\partial^2}{\partial\varepsilon^2}\big[\mf Q(\tilde{\mf x})\cdot(\tilde{\mf x}_u\wedge\tilde{\mf x}_v)\big]
	\Big|_{\varepsilon=0}\\[1ex]
	&& \hspace*{4ex} =2W\varphi^2\big[\nabla\ml H(\mf x)\cdot\mf N-2\ml H(\mf x)^2\big]+2\ml H(\mf x)\mf y\cdot(\mf x_u\wedge\mf y_v+\mf y_u\wedge
	\mf x_v)\\[1ex]
	&& \hspace*{6.5ex} -2\ml H(\mf x)\big[\varphi(\mf x_u\cdot\mf y_u)+\varphi(\mf x_v\cdot\mf y_v)+(\mf x_u\cdot\mf y)
	\varphi_u+(\mf x_v\cdot\mf y)\varphi_v\big]\\[1ex]
	&& \hspace*{6.5ex} +2\big[\varphi\ml H(\mf x)(\mf x_u\cdot\mf y)\big]_u	+2\big[\varphi\ml H(\mf x)(\mf x_v\cdot\mf y)\big]_v
	+2\ml H(\mf x)\mf z\cdot(\mf x_u\wedge\mf x_v)\\[1ex]
	&& \hspace*{6.5ex} +\big[\big(D\mf Q(\mf x)\mf y\big)\cdot(\mf y\wedge\mf x_v)\big]_u+\big[\mf Q(\mf x)\cdot(\mf z\wedge\mf x_v)\big]_u
	+ [\mf Q(\mf x)\cdot(\mf y\wedge\mf y_v)\big]_u\\[1ex] 
	&& \hspace*{6.5ex} + \big[\big(D\mf Q(\mf x)\mf y\big)\cdot(\mf x_u\wedge\mf y)\big]_v+ \big[\mf Q(\mf x)\cdot(\mf x_u\wedge\mf z)\big]_v 
	+[\mf Q(\mf x)\cdot(\mf y_u\wedge\mf y)\big]_v
	\end{eqnarray*}
on $B'$.
\end{prop}

\begin{proof}
Using (\ref{g0.4}) and the general relation
	\bee\label{g1.18}
	[\mf M\mf a]\cdot(\mf b\wedge\mf c)+ \mf a\cdot([\mf M\mf b]\wedge\mf c)+ \mf a\cdot(\mf b\wedge[\mf M\mf c])=(\mbox{tr}\,\mf M)[\mf a\cdot
	(\mf b\wedge\mf c)]
	\ee
for arbitrary vectors $\mf a,\mf b,\mf c\in\mb R^3$ and matrices $\mf M\in\mb R^{3\times 3}$ with trace tr$\,\mf M$, we first compute
	$$\begin{array}{l}
	\hspace*{-2ex}\ds\frac\partial{\partial\varepsilon}\big[\mf Q(\tilde{\mf x})\cdot(\tilde{\mf x}_u\wedge\tilde{\mf x}_v)\big]\\[1.5ex]
	\hspace*{2ex} = [D\mf Q(\tilde{\mf x})\tilde{\mf x}_\varepsilon]\cdot(\tilde{\mf x}_u\wedge\tilde{\mf x}_v)+\mf Q(\tilde{\mf x})
	\cdot(\tilde{\mf x}_\varepsilon\wedge\tilde{\mf x}_v)_u+\mf Q(\tilde{\mf x})\cdot(\tilde{\mf x}_u\wedge\tilde{\mf x}_\varepsilon)_v\\[1.5ex]
	\hspace*{2ex} = 2{\ml H}(\tilde{\mf x})\tilde{\mf x}_\varepsilon\cdot(\tilde{\mf x}_u\wedge\tilde{\mf x}_v)+\big[\mf Q(\tilde{\mf x})\cdot
	(\tilde{\mf x}_\varepsilon\wedge\tilde{\mf x}_v)\big]_u+\big[\mf Q(\tilde{\mf x})\cdot(\tilde{\mf x}_u\wedge\tilde{\mf x}_\varepsilon)\big]_v
	\end{array}$$
on $B^+$. Having (\ref{g1.5+}) and (\ref{g1.6}) in mind, a second differentiation yields at $\varepsilon=0$:
	\bee\label{g1.19}
	\begin{array}{rcl}
	\!\!\!\ds\frac{\partial^2}{\partial\varepsilon^2}\big[\mf Q(\tilde{\mf x})\cdot(\tilde{\mf x}_u\wedge\tilde{\mf x}_v)\big]
	\Big|_{\varepsilon=0}\!\!\! & = & \!\!2\big[\nabla\ml H(\mf x)\cdot\mf y\big]\mf y\cdot(\mf x_u\wedge\mf x_v)+2\ml H(\mf x)\mf z\cdot
	(\mf x_u\wedge\mf x_v)\\[1.5ex]
	&& \!\!+2\ml H(\mf x)\mf y\cdot(\mf x_u\wedge\mf y_v+\mf y_u\wedge\mf x_v)\\[1.5ex]
	&& \!\!+ \big[\big(D\mf Q(\mf x)\mf y\big)\cdot(\mf y\wedge\mf x_v)\big]_u+\big[\mf Q(\mf x)\cdot(\mf z\wedge\mf x_v)\big]_u\\[1.5ex]
	&& \!\!+ [\mf Q(\mf x)\cdot(\mf y\wedge\mf y_v)\big]_u + \big[\big(D\mf Q(\mf x)\mf y\big)\cdot(\mf x_u\wedge\mf y)\big]_v\\[1.5ex]
	&& \!\!+ \big[\mf Q(\mf x)\cdot(\mf x_u\wedge\mf z)\big]_v +[\mf Q(\mf x)\cdot(\mf y_u\wedge\mf y)\big]_v.
	\end{array}
	\ee
Writing again $\mf y=\lambda^j\mf x_{u^j}+\varphi\mf N$ on $B'$ with $\lambda^j=W^{-1}\mf x_{u^j}\cdot\mf y$ and employing (\ref{g0.8}), the assertion follows from (\ref{g1.19}) and the identity
	\begin{eqnarray*}
	&& \hspace*{-4ex}2\big[\nabla\ml H(\mf x)\cdot\mf y\big]\mf y\cdot(\mf x_u\wedge\mf x_v)\\[1ex]
	&& \hspace*{8ex} = 2W\varphi^2\nabla\ml H(\mf x)\cdot\mf N+2\varphi\lambda^jW\ml H
	(\mf x)_{u^j}\\[1ex]
	&& \hspace*{8ex} = 2W\varphi^2\nabla\ml H(\mf x)\cdot\mf N+2\big[\varphi\ml H(\mf x)(\mf x_u\cdot\mf y)\big]_u
	+2\big[\varphi\ml H(\mf x)(\mf x_v\cdot\mf y)\big]_v\\[1ex]
	&& \hspace*{10.5ex} -2\ml H(\mf x)\big[\varphi(\mf x_u\cdot\mf y_u)+\varphi(\mf x_v\cdot\mf y_v)+(\mf x_u\cdot\mf y)
	\varphi_u+(\mf x_v\cdot\mf y)\varphi_v\big]\\[1ex]
	&& \hspace*{10.5ex} -4W\varphi^2\ml H(\mf x)^2\\[-6ex]
	\end{eqnarray*}
on $B'$.\end{proof}
	
As already announced, the right-hand sides in the results of Propositions \ref{p1} and \ref{p2} can be extended continuously onto $B^+\cup I$, according to Lemma\,\ref{l2}. Hence we can compute the second variation via the divergence theorem for \emph{any} admissible one-parameter family $\tilde{\mf x}(\cdot,\varepsilon)$ with direction $\mf y\in C^2_c(B^+\cup I,\mb R^3)$ satisfying (\ref{g1.1}). Nevertheless, we concentrate on directions of the form
	\bee\label{g1.20}
	\mf y(w):=\frac{\varphi(w)}{1+\mf Q(\mf x(w))\cdot\mf N(w)}\big[\mf Q(\mf x(w))+\mf N(w)\big],
	\ee
with some function $\varphi\in C^2_c(B^+\cup I)$. Note that $\mf y$ is well-defined according to assumption (\ref{g0.4}), belongs to $C^2_c(B^+\cup I,\mb R^3)$, and satisfies $\mf y\cdot\mf N\equiv\varphi$ as well as (\ref{g1.1}); for the latter, see Remark\,\ref{r1}. 

\begin{definition}\label{d4}
For given $\varphi\in C^2_c(B^+\cup I)$ we define $\mf y\in C^2_c(B^+\cup I,\mb R^3)$ by (\ref{g1.20}) and consider the admissible perturbation $\tilde{\mf x}(\cdot,\varepsilon)$ with direction $\mf y$. Then we set 
	$$\delta^2A_{\mf Q}(\mf x,\varphi):=\frac{d^2}{d\varepsilon^2}A_{\mf Q}\big(\tilde{\mf x}(\cdot,\varepsilon))\big)\Big|_{\varepsilon=0}$$
for the \emph{second variation of $A_{\mf Q}(\mf x)$ with dilation $\varphi$}.
\end{definition}

In order to compute $\delta^2A_{\mf Q}(\mf x,\varphi)$, we introduce the curvature of the cylindrical support surface $S$ defined by
	\bee\label{g1.21+}
	\kappa(\mf p):=-\big(\sigma''(s),0\big)\cdot\mf n(\mf p)\quad\mbox{for}\ \,\mf p\in\{\sigma(s)\}\times\mb R,\ s\in[0,s_0],
	\ee
compare Section\,2. Note that, due to the cylindrical structure of $S$, we have the relation
	\bee\label{g1.21}
	\big[D\mf n(\mf p)\bs\zeta_1]\cdot\bs\zeta_2=\kappa(\mf p)\big[\bs\zeta_1\cdot\mf t(\mf p)\big]\big[\bs\zeta_2\cdot\mf t(\mf p)\big]
	\quad\mbox{for all}\ \bs\zeta_1,\bs\zeta_2\in T_{\mf p}S,\ \mf p\in S,
	\ee
interpreting $D\mf n$ as the Weingarten map of $S$.

\begin{lemma}\label{l3}
Let $\mf x\in\ml C_\mu(\Gamma, S;\overline{Z})$, $\mu\in(0,1)$, be a stationary $\ml H$-surface w.r.t.~$E_{\mf Q}$ and let $\varphi\in C^2_c(B^+\cup I)$ be chosen. Setting
	\bee\label{g1.22}
	q(w):=\big[2\ml H(\mf x(w))^2-K(w)-\nabla\ml H(\mf x(w))\cdot\mf N(w)\big]W(w),\quad w\in B^+\cup I,
	\ee
we then have 
	\bee\label{g1.23}
	\begin{array}{rcl}
	\delta^2A_{\mf Q}(\mf x,\varphi)\!\!\! &=& \!\!\!\ds\iint\limits_{B^+}\Big\{|\nabla\varphi|^2-2q\varphi^2\Big\}\,du\,dv
	+\int\limits_I\varphi^2\frac{\mf N_v\cdot\mf Q(\mf x)}{1+\mf Q(\mf x)\cdot\mf N}\,du\\[4ex]
	&& \!\!\!\!\ds +\int\limits_I\varphi^2\bigg\{\frac{\big[D\mf Q(\mf x)\big(\mf Q(\mf x)+\mf N\big)\big]\cdot\big[\mf x_v+\mf Q(\mf x)\wedge \mf x_u
	\big]}{(1+\mf Q(\mf x)\cdot\mf N)^2}\\[3ex]
	&& \ds\hspace{1.9ex}+\frac{\kappa(\mf x)\big[\big(\mf x_v+\mf Q(\mf x)\wedge\mf x_u\big)\cdot\mf n(\mf x)\big]\big[\big(\mf Q(\mf x)+\mf N\big)
	\cdot\mf t(\mf x)\big]^2}{(1+\mf Q(\mf x)\cdot\mf N)^2}
	\bigg\}du.
	\end{array}
	\ee
\end{lemma}

\begin{proof}
We add the results of Propositions \ref{p1} and \ref{p2} obtaining
	\begin{eqnarray*}
	\hspace*{-3ex}&&\frac{\partial^2}{\partial\varepsilon^2}\big(|\tilde{\mf x}_u\wedge\tilde{\mf x}_v|+\mf Q(\tilde{\mf x})
	\cdot\tilde{\mf x}_u\wedge\tilde{\mf x}_v\big)\Big|_{\varepsilon=0}\\[1ex]
	&& \hspace*{10ex} = |\nabla\varphi|^2-2q\varphi^2-\big[\varphi(\mf N_u\cdot\mf y)\big]_u-\big[\varphi(\mf N_v\cdot\mf y)\big]_v\\[1ex]
	&& \hspace*{12.5ex} + \big[\big(D\mf Q(\mf x)\mf y\big)\cdot(\mf y\wedge\mf x_v)\big]_u + \big[\big(D\mf Q(\mf x)\mf y\big)
	\cdot(\mf x_u\wedge\mf y)\big]_v\\[1ex]
	&& \hspace*{12.5ex} + \big[\big(\mf Q(\mf x)+\mf N\big)\cdot(\mf y\wedge\mf y_v)\big]_u + \big[\big(\mf Q(\mf x)+\mf N\big)
	\cdot(\mf y_u\wedge\mf y)\big]_v\\[1ex]
	&& \hspace*{12.5ex} + \big[\mf z\cdot\big(\mf x_u+\mf x_v\wedge\mf Q(\mf x)\big)\big]_u + \big[\mf z\cdot\big(\mf x_v+\mf Q(\mf x)\wedge
	\mf x_u\big)\big]_v.
	\end{eqnarray*}
Having $\mf y\parallel(\mf Q(\mf x)+\mf N)$ on $I$ in mind, the divergence theorem yields
	\bee\label{g1.24}
	\begin{array}{rcl}
	\delta^2A_{\mf Q}(\mf x,\varphi) & = & \ds \iint\limits_{B^+}\frac{\partial^2}{\partial\varepsilon^2}\big(|\tilde{\mf x}_u\wedge\tilde{\mf x}_v|
	+\mf Q(\tilde{\mf x})\cdot\tilde{\mf x}_u\wedge\tilde{\mf x}_v\big)\Big|_{\varepsilon=0}\\[3.5ex]
	& = & \ds\iint\limits_{B^+}\Big\{|\nabla\varphi|^2-2q\varphi^2\Big\}\,du\,dv +\int\limits_I\varphi(\mf N_v\cdot\mf y)\,du\\[3ex]
	&& \ds -\int\limits_I\Big\{\big(D\mf Q(\mf x)\mf y\big)\cdot(\mf x_u\wedge\mf y)+\mf z\cdot\big(\mf x_v+\mf Q(\mf x)\wedge\mf x_u\big)\Big\}\,du.
	\end{array}
	\ee
Due to the special choice (\ref{g1.20}) of $\mf y$, the first three terms on the right-hand side of (\ref{g1.24}) are identical with those in the announced relation (\ref{g1.23}). In order to identify the fourth terms of (\ref{g1.23}) and (\ref{g1.24}), we recall Lemma\,\ref{l2}\,(i) and deduce
	\bee\label{g1.25}
	\mf z\cdot\big(\mf x_v+\mf Q(\mf x)\wedge\mf x_u\big)=\big(\mf z\cdot\mf n(\mf x)\big)\big[\big(\mf x_v+\mf Q(\mf x)\wedge\mf x_u\big)\cdot
	\mf n(\mf x)\big]\quad\mbox{on}\ I.
	\ee
Similar to \cite{hildefritz} p.\,431, we compute $\mf z\cdot\mf n(\mf x)$ on $I$: Since $\tilde{\mf x}(w,\varepsilon)\in S$ holds for all $w\in I$ and $\varepsilon\in(-\varepsilon_0,\varepsilon_0)$, we have $\frac\partial{\partial\varepsilon}\tilde{\mf x}(w,\varepsilon)\cdot\mf n(\tilde{\mf x}(w,\varepsilon))=0$ and, consequently,
	$$\frac{\partial^2}{\partial\varepsilon^2}\tilde{\mf x}(w,\varepsilon)\cdot\mf n(\tilde{\mf x}(w,\varepsilon))
	+\frac\partial{\partial\varepsilon}\tilde{\mf x}(w,\varepsilon)\cdot\Big[D\mf n(\tilde{\mf x}(w,\varepsilon))
	\frac\partial{\partial\varepsilon}\tilde{\mf x}(w,\varepsilon)\Big]=0$$
for $w\in I$ and $\varepsilon\in(-\varepsilon_0,\varepsilon_0)$. For $\varepsilon=0$ we employ (\ref{g1.21}) and infer 
	$$\mf z\cdot\mf n(\mf x)=-\kappa(\mf x)\big[\mf y\cdot\mf t(\mf x)\big]^2\quad\mbox{on}\ I.$$
Together with (\ref{g1.25}), we arrive at 
	$$\mf z\cdot\big(\mf x_v+\mf Q(\mf x)\wedge\mf x_u\big)=-\kappa(\mf x)\big[\big(\mf x_v+\mf Q(\mf x)\wedge\mf x_u\big)\cdot
	\mf n(\mf x)\big]\big[\mf y\cdot\mf t(\mf x)\big]^2\quad\mbox{on}\ I.$$
Putting this relation into (\ref{g1.24}), proves the assertion.
\end{proof}

\begin{remark}\label{r3}
By a standard approximation argument, dilations $\varphi\in H_2^1(B^+)\cap C_c^0(B^+\cup I)$ are admissible in the second variation $\delta^2A_{\mf Q}(\mf x,\varphi)$ due to formula (\ref{g1.23}).
\end{remark}

\begin{definition}\label{d5}
A partially free $\ml H$-surface $\mf x\in C_\mu(\Gamma,S;\overline{Z})$ with $\delta^2A_{\mf Q}(\mf x,\varphi)\ge0$ for any dilation $\varphi\in H_2^1(B^+)\cap C_c^0(\overline{B^+})$ is called \emph{stable}.
\end{definition}

%%%%%%%%%%%%%%%%%%%%%%%%%%%%%%%%%%%%%%%%%%%%%%%%%%%%%%%%%%%%%%%%%%%%%%%%%%%%%%%%%%%%%%%%%%%%%%%%%%%%%%%%%%%%%%%%%%%%%%%%%%%%%%%%%%%%%%%%%%%

\setcounter{equation}{0}
\section{Boundary condition for the surface normal and proof of the theorem}
In order to deduce the crucial relation $N^3>0$ on $\overline{B^+}$ for the third component of the surface normal of our stable $\ml H$-surface, we will combine formula (\ref{g1.23}) with the following boundary condition:

\begin{lemma}\label{l4}
Let the assumptions of Theorem\,\ref{t1} be satisfied and let a stationary $\ml H$-surface $\mf x\in\ml C_\mu(\Gamma, S;\overline{Z})$, $\mu\in(0,1)$, be given. Then, the third component $N^3$ of the surface normal of $\mf x$ fulfills the boundary condition	
	\bee\label{g4.1}
	\begin{array}{rcl}
	N_v^3 & = & \ds\bigg\{\frac{\mf N_v\cdot\mf Q(\mf x)}{1+\mf Q(\mf x)\cdot\mf N}+\frac{\big[D\mf Q(\mf x)\big(\mf Q(\mf x)+\mf N\big)\big]\cdot
	\big[\mf x_v+\mf Q(\mf x)\wedge\mf x_u\big]}{(1+\mf Q(\mf x)\cdot\mf N)^2}\\[3ex]
	&& +\ds\frac{\kappa(\mf x)\big[\big(\mf x_v+\mf Q(\mf x)\wedge\mf x_u\big)\cdot\mf n(\mf x)\big]\big[\big(\mf Q(\mf x)+\mf N\big)\cdot\mf t(\mf x)
	\big]^2}{(1+\mf Q(\mf x)\cdot\mf N)^2}\bigg\}N^3
	\quad\mbox{on}\ I,
	\end{array}
	\ee
where $\mf t$, $\mf n$, and $\kappa$ were defined in (\ref{g0.1}), (\ref{g1.21+}).
\end{lemma}

\noindent{\it Proof.} 
\begin{enumerate}
\item
From (\ref{g0.8}) and Lemma\,\ref{l2}\,(iv) we get the well known relations
	\bee\label{g4.2}
	\mf N_u=\mf N\wedge\mf N_v-2\ml H(\mf x)\mf x_u,\quad\mf N_v=-\mf N\wedge\mf N_u-2\ml H(\mf x)\mf x_v\quad\mbox{on}\ B^+\cup I.
	\ee
Writing $\ml H=\ml H(\mf x)$, $\mf Q=\mf Q(\mf x)$, $\kappa=\kappa(\mf x)$ etc.~and employing (\ref{g4.2}) as well as (\ref{g0.15}), we compute
	$$\begin{array}{l}
	(\mf N_v\cdot\mf Q)N^3\,=\,\big\{[(\mf Q+\mf N)\cdot\mf N_v]\mf N\big\}\cdot\mf e_3\\[2ex]
	\hspace{6ex}= \,-\Big\{(\mf N\wedge\mf N_v)\wedge (\mf Q+\mf N)-[\mf N\cdot(\mf Q+\mf N)]\mf N_v\Big\}\cdot\mf e_3\\[2ex]
	\hspace{6ex} = \,-\Big\{\mf N_u\wedge(\mf Q+\mf N)+2\ml H\mf x_u\wedge(\mf Q+\mf N)-[1+(\mf Q\cdot\mf N)]\mf N_v\Big\}\cdot\mf e_3\\[2ex]
	\hspace{6ex} = \,(\mf N\wedge\mf e_3)_u\cdot(\mf Q+\mf N)+[1+(\mf Q\cdot\mf N)]N_v^3\quad\mbox{on}\ I.
	\end{array}$$
Consequently, the asserted relation (\ref{g4.1}) is equivalent to the identity
	\bee\label{g4.3}
	\begin{array}{rl}
	\!\!(\mf N\wedge\mf e_3)_u\cdot(\mf Q+\mf N)\,= & \!\!-\ds\Big\{[D\mf Q(\mf Q+\mf N)]\cdot(\mf x_v+\mf Q\wedge\mf x_u)\\[1ex]
	& \!\!\ds+\,\kappa[(\mf x_v+\mf Q\wedge\mf x_u)\cdot\mf n][(\mf Q+\mf N)\cdot\mf t]^2\Big\}\frac{N^3}{1+\mf Q\cdot\mf N}
	\end{array}
	\ee
on $I$.
\item
Next, we manipulate the left-hand side of (\ref{g4.3}): Having (\ref{g0.15}) in mind, we find
	$$(\mf Q+\mf N)\wedge\mf e_3=(\mf Q+\mf N)\wedge(\mf n\wedge\mf t)=[(\mf Q+\mf N)\cdot\mf t]\mf n\quad\mbox{on}\ I.$$
Together with (\ref{g1.21}), we infer
	\bee\label{g4.4}
	\begin{array}{rcl}
	[(\mf Q+\mf N)\wedge\mf e_3]_u\cdot(\mf Q+\mf N) & = & [(\mf Q+\mf N)\cdot\mf t]\big\{[(D\mf n)\mf x_u]\cdot(\mf Q+\mf N)\big\}\\[1.5ex]
	& = & \kappa[(\mf Q+\mf N)\cdot\mf t]^2(\mf x_u\cdot\mf t)\quad\mbox{on}\ I.
	\end{array}
	\ee
On the other hand, we calculate
	$$\begin{array}{rcl}
	\!\!(\mf x_u\cdot\mf t)(1+\mf Q\cdot\mf N)\! & = & \!(\mf x_u\cdot\mf t)[\mf N\cdot(\mf Q+\mf N)]\\[1.5ex]
	& = & \![\mf x_u\wedge(\mf Q+\mf N)]\cdot(\mf t\wedge\mf N)-(\mf x_u\cdot\mf N)[\mf t\cdot(\mf Q+\mf N)]\\[1.5ex]
	& = & \!\big\{[\mf x_u\wedge(\mf Q+\mf N)]\cdot\mf n\big\}[\mf n\cdot(\mf t\wedge\mf N)]\\[1.5ex]
	& = & \!-[(\mf x_v+\mf Q\wedge\mf x_u)\cdot\mf n]N^3\quad\mbox{on}\ I
	\end{array}$$
or, equivalently,
	\bee\label{g4.5}
	\mf x_u\cdot\mf t=-\frac{N^3}{1+\mf Q\cdot\mf N}[(\mf x_v+\mf Q\wedge\mf x_u)\cdot\mf n]\quad\mbox{on}\ I.
	\ee
From (\ref{g4.4}) and (\ref{g4.5}) we now deduce
	\bee\label{g4.6}
	\begin{array}{rl}
	\hspace{-1ex}(\mf N\wedge\mf e_3)_u\cdot(\mf Q+\mf N)\!\!\! & =\,[(\mf Q+\mf N)\wedge\mf e_3]_u\!\cdot\!(\mf Q+\mf N)-(\mf Q\wedge\mf e_3)_u
	\!\cdot\!(\mf Q+\mf N)\\[2ex]
	& =\,-\ds\kappa[(\mf x_v+\mf Q\wedge\mf x_u)\cdot\mf n][(\mf Q+\mf N)\cdot\mf t]^2\frac{N^3}{1+\mf Q\cdot\mf N}\\[3ex]
	& \quad\,-(\mf Q\wedge\mf e_3)_u\cdot(\mf Q+\mf N)\quad\mbox{on}\ I.
	\end{array}
	\ee
By inserting (\ref{g4.6}) into (\ref{g4.3}), the claimed relation (\ref{g4.1}) becomes equivalent to 
	\bee\label{g4.7}
	(\mf Q\wedge\mf e_3)_u\cdot(\mf Q+\mf N)=[D\mf Q(\mf Q+\mf N)]\cdot(\mf x_v+\mf Q\wedge\mf x_u)\frac{N^3}{1+\mf Q\cdot\mf N}\quad\mbox{on}\ I.
	\ee
\item
In the next step, we observe that (\ref{g4.7}) is equivalent to the identity
	\bee\label{g1.32}
	[(D\mf Q)\mf x_u]\cdot[\mf e_3\wedge(\mf Q+\mf N)]+\mf x_u\cdot\big\{\mf e_3\wedge[(D\mf Q)(\mf Q+\mf N)]\big\}=0\quad\mbox{on}\ I.
	\ee
Indeed, the left hand side of (\ref{g4.7}) can be written as
	$$(\mf Q\wedge\mf e_3)_u\cdot(\mf Q+\mf N)=\big\{[(D\mf Q)\mf x_u]\wedge\mf e_3\big\}\cdot(\mf Q+\mf N)=[(D\mf Q)\mf x_u]\cdot[\mf e_3\wedge(\mf Q+\mf N)],$$
whereas, using Lagrange identity and boundary condition (\ref{g0.13}), we compute on the right hand side
	$$\begin{array}{l}
	[D\mf Q(\mf Q+\mf N)]\cdot(\mf x_v+\mf Q\wedge\mf x_u)N^3\\[1.5ex]
	\hspace{18ex} =[(\mf x_v+\mf Q\wedge\mf x_u)\wedge\mf N]\cdot\big\{[D\mf Q(\mf Q+\mf N)]\wedge\mf e_3\big\}\\[1.5ex]
	\hspace{18ex} = (1+\mf Q\cdot\mf N)\,\mf x_u\cdot\big\{[D\mf Q(\mf Q+\mf N)]\wedge\mf e_3\big\}\quad\mbox{on}\ I.
	\end{array}$$
This proves the claimed equivalence.
\item
It remains to prove (\ref{g1.32}). Applying the relation (\ref{g1.18}) with $\mf a=\mf x_u$, $\mf b=\mf e_3$, $\mf c=\mf Q+\mf N$, and $\mf M=D\mf Q$, we obtain
	$$\begin{array}{l}
	[(D\mf Q)\mf x_u]\cdot[\mf e_3\wedge(\mf Q+\mf N)]+\mf x_u\cdot\big\{\mf e_3\wedge[(D\mf Q)(\mf Q+\mf N)]\big\}\\[1.5ex]
	\hspace{5ex}=-\mf x_u\cdot\big\{[(D\mf Q)\mf e_3]\wedge(\mf Q+\mf N)\big\}+(\mbox{tr}\,D\mf Q)\big\{\mf x_u\cdot[\mf e_3\wedge(\mf Q+\mf N)]\big\}\\[1.5ex]
	\hspace{5ex}=[(D\mf Q)\mf e_3]\cdot[\mf x_u\wedge(\mf Q+\mf N)]\quad\mbox{on}\ I,
	\end{array}$$
where we also used $\mf Q+\mf N\parallel T_{\mf x}S$. For the same reason, $\mf x_u\wedge(\mf Q+\mf N)$ is normal to $S$ along $I$ and, as a consequence, the right hand side of the above identity vanishes. Indeed, we have
	$$[D\mf Q(\mf p)\mf e_3]\cdot\mf n(\mf p)=\Big[\frac\partial{\partial p^3}\mf Q(\mf p)\Big]\cdot\mf n(\mf p)=\frac{\partial}{\partial p^3}\big[\mf Q(\mf p)\cdot\mf n(\mf p)\big]=0\quad\mbox{on}\ 
	S,$$
by assumption. This completes the proof of (\ref{g1.32}), and (\ref{g4.1}) is confirmed.

\hfill q.e.d.
\end{enumerate}

\vspace{1ex}
We are now able to give the

\vspace{2ex}
\begin{proof}[Proof of Theorem\,\ref{t1}.]
\begin{enumerate}
\item
According to Lemma\,\ref{l2}\,(iv), the surface normal $\mf N=(N^1,N^2,N^3)$ of $\mf x$ belongs to $C^{2,\alpha}(B^+)\cap C^{1,\alpha}(\overline{B^+}\setminus\{-1,+1\})\cap C^0(\overline{B^+})$. In addition, the inclusion $f(\overline{B^+})\subset\overline G$ and the $\frac1R$-convexity of $G$ imply $N^3>0$ on $J\setminus\{-1,+1\}$ as was shown in \cite{fritz} Satz\,2. The behaviour of the surface normal near the corner points $\pm1$ was studied in \cite{mueller4} Theorem\,5.4; the applicability of the cited result follows -- after reflecting $S$ and rotating appropriately in $\mb R^3$ -- from the assumption $|(\mf Q\cdot\mf n)(\mf p_j)|<\cos\alpha_j\le\cos\gamma_j$ for $j=1,2$, where $\gamma_j$ denote the angles between $\Gamma$ and $S$ at $\mf p_j$ ($j=1,2$). In particular, $N^3(\pm1)$ cannot vanish and, by continuity, we infer $N^3(\pm1)>0$. Consequently, the dilation $\omega:=(N^3)^-=\max\{0,-N^3\}\in C_c^0(B^+\cup I)\cap H^1_2(B^+)$ is admissible in the second variation of $A_{\mf Q}(\mf x)$. Writing $\omega^2=-\omega\,N^3$ and $|\nabla\omega|^2=-\nabla\omega\cdot\nabla N^3$, we obtain from Lemmas \ref{l3} and \ref{l4}: 
	\begin{eqnarray*}
	&& \delta^2A_{\mf Q}(\mf x,\omega) = \iint\limits_{B^+}\{|\nabla\omega|^2-2q\omega^2\}\,du\,dv-\int\limits_I\omega N_v^3\,du\\[1ex]
	&& \hspace{5ex}= -\iint\limits_{B^+}\big\{\mbox{div}(\omega\nabla N^3)+\omega(\Delta N^3+2q N^3)\}\,du\,dv-\int\limits_I\omega N_v^3\,du\\[1ex]
	&& \hspace{5ex}= \iint\limits_{B^+}\omega(\Delta N^3+2qN^3)\,du\,dv = -2\iint\limits_{B^+}\omega \ml H_{p^3}(\mf x)W\,du\,dv \le0,
	\end{eqnarray*}
where we have applied Gauss' theorem, equation (\ref{g0.13+}), and assumption (\ref{g0.10}) in the last line. The stability of $\mf x$ thus yields $\delta^2A_{\mf Q}(\mf x,\omega)=0$.
\item
Now we choose $\xi\in C^\infty_c(B^+)$ arbitrarily. Then also $\omega+\varepsilon\xi$ is admissible in $\delta^2A_{\mf Q}(\mf x,\cdot)$ for any $\varepsilon\in\mb R$. The function $\Xi(\varepsilon):=\delta^2A_{\mf Q}(\mf x,\omega+\varepsilon\xi)$ depends smoothly on $\varepsilon\in\mb R$ and satisfies $\Xi\ge0$ as well as $\Xi(0)=0$. Consequently, we have $\Xi'(0)=0$, which means
	$$\iint\limits_{B^+}\{\nabla\omega\cdot\nabla\xi-2q\omega\xi\}\,du\,dv=0\quad\mbox{for any}\ \xi\in C_c^\infty(B^+),$$
according to formula (\ref{g1.23}). From $\omega=0$ near $J$, we conclude $\omega\equiv0$  by means of the weak Harnack inequality. Hence, we have $N^3\ge0$ in $\overline{B^+}$. Due to assumption (\ref{g0.10}) and equation (\ref{g0.13+}), we further have $\Delta N^3+2qN^3\le0$ in $B^+$. Therefore, Harnack's inequality, in conjunction with $N^3>0$ near $J$, yields $N^3>0$ in $B^+\cup J$. Finally, we have $N^3>0$ on $I$ and hence everywhere on the closed half disc $\overline{B^+}$. Indeed, if $N^3(w_0)=0$ would be true for some point $w_0\in I$, relation (\ref{g4.1}) would imply $N_v^3(w_0)=0$. But this is impossible due to Hopf's boundary point lemma.

\item
Since we have no branch points on $\partial B^+\setminus\{-1,+1\}$ according to Lemma \ref{l2}\,(iii), the relation $N^3>0$ on $\partial B^+$ implies $x^1_ux^2_v-x^2_ux^1_v>0$ on $\partial{B^+}\setminus\{-1,+1\}$. Consequently, the projection $f=\pi(\mf x)=(x^1,x^2):\overline{B^+}\to\mb R^2$ maps $\partial B^+$ topologically and positively oriented onto $\partial G$. As in \cite{fritz} Hilfs\-satz\,7, an index argument now shows that $f:\overline{B^+}\to\overline G$ is a homeomorphism, $\mf x$ has no branch points in $\overline{B^+}\setminus\{-1,+1\}$, and $J_f>0$ is satisfied in $\overline{B^+}\setminus\{-1,+1\}$. By the inverse mapping theorem and the regularity of $\mf x$, the mapping $f:\overline G\to\overline{B^+}$ belongs to $C^2(\overline G\setminus\{p_1,p_2\})\cap C^0(\overline G)$, where we abbreviated $p_j=\pi(\mf p_j)$, $j=1,2$. 

Now we consider $\zeta:=x^3\circ f^{-1}\in C^2(\overline{G}\setminus\{p_1,p_2\})\cap C^0(\overline G)$. Since we have $(x^1,x^2,\zeta(x^1,x^2))=\mf x\circ f^{-1}(x^1,x^2)$, $\zeta$ is the desired graph representation over $\overline G$ satisfying the differential equation (\ref{g0.12-1}) and the second boundary condition in (\ref{g0.12-2}). In addition, we compute
	\begin{eqnarray*}
	\psi(\mf x) &=& \mf Q(\mf x)\cdot\mf n(\mf x)\ \stackrel{(\ref{g0.15})}{=}\ -\mf N\cdot\mf n(\mf x)\\[1ex]
	&=& \frac1{\sqrt{1+|\nabla\zeta|^2}}(\zeta_{x^1},\zeta_{x^2},-1)\cdot(\nu(\mf x),0)\\[1ex]
	&=& \frac{\nabla\zeta\cdot\nu(\mf x)}{\sqrt{1+|\nabla\zeta|^2}},\quad\mf x=(x^1,x^2,\zeta(x^1,x^2)),\quad(x^1,x^2)\in\Sigma.
	\end{eqnarray*}
Hence, $\zeta$ is a solution of the boundary value problem (\ref{g0.12-1}), (\ref{g0.12-2}), and standard elliptic theory yields $\zeta\in C^{3,\alpha}(G)\cap C^{2,\alpha}(\overline G\setminus\{p_1,p_2\})$ according to the regularity assumptions on $\mf Q$, $\ml H$, $S$, and $\Gamma$. This completes the proof.
\end{enumerate}
\end{proof}

We finally give an example of how to apply Theorem\,\ref{t1} to the existence question for the mixed boundary value problem (\ref{g0.12-1}), (\ref{g0.12-2}).

\begin{corollar}\label{c2}
Let $G\subset B_R:=\{(x^1,x^2)\in\mb R^2\,:\ |(x^1,x^2)|<R\}$ be a $\frac1R$-convex domain with boundary $\partial G=\ul\Gamma\cup\Sigma$, where $\ul\Gamma,\Sigma\in C^3$ are closed Jordan arcs, which satisfy $\ul\Gamma\cap\Sigma=\{\pi_1,\pi_2\}$ and which meet with interior angles $\alpha_j\in(0,\frac\pi2]$ w.r.t.~$G$ at the distinct points $\pi_j$ ($j=1,2$). In addition, assume that $\Sigma$ can be written as a graph
	$$\Sigma=\big\{(x^1,x^2))\in\mb R^2\,:\ x^2=g(x^1),\ a\le x^1\le b\big\},\qquad -R<a<b<R,$$
with some function $g\in C^3([-R,R])$. Moreover, let $\ml H\in C^{1,\alpha}(\overline{B_R})$, $\psi\in C^{1,\alpha}(\Sigma)$ and $\gamma\in C^3(\ul\Gamma)$ be given functions and abbreviate $h_0:=\sup_{B_R}|H|$, $\psi_0:=\sup_\Sigma|\psi|$, $g_0:=\sup_{[-R,R]}|g'|$. Finally, suppose the conditions
	\bee\label{g4.8}
	4Rh_0+\psi_0\sqrt{1+g_0^2}<1,\qquad|\psi(\pi_j)|<\cos\alpha_j,\quad j=1,2,
	\ee
to be satisfied. Then, the boundary value problem (\ref{g0.12-1}), (\ref{g0.12-2}) has a unique solution $\zeta\in C^{3,\alpha}(G)\cap C^{2,\alpha}(\overline G\setminus\{\pi_1,\pi_2\})\cap C^0(\overline G)$.
\end{corollar}

\begin{remark}
Note that the prescribed mean curvature function $\ml H$ in Corollary\,\ref{c2} does not depend on the hight $p^3$. If one wants to allow such a dependence, one has to use estimates for the length of the free trace as given in \cite{mueller6}; see \cite{mueller} sec.~6 for a description of the required arguments.
\end{remark}

\begin{proof}[Proof of Corollary\,\ref{c2}.]
We assume w.l.o.g.~that the exterior normal $\nu$ w.r.t.~$G$ is given by $\nu=(1+(g')^2)^{-\frac12}(g',-1)$ along $\Sigma$ and set
	$$Q_2(p^1,p^2):=2\int_{g(p^1)}^{p^2}H(p^1,\eta)\,d\eta-\psi(p^1,g(p^1))\sqrt{1+g'(p^1)},\quad(p^1,p^2)\in\overline{B_R}.$$
We use the notations $Z=B_R\times\mb R$, $\Gamma=\mbox{graph}\,\varphi$, $S=\Sigma\times\mb R,\mf n=(\nu,0)$, $\ldots$ from above and set $\mf Q(\mf p):=(0,Q_2(p^1,p^2),0)$ for $\mf p=(p^1,p^2,p^3)\in\overline Z$. Then, $\mf Q$ belongs to $C^{1,\alpha}(\overline Z,\mb R^3)$ and satisfies 
	$$\mbox{div}\,\mf Q=Q_{2,p^2}=2\ml H\quad\mbox{in}\ \overline Z,\qquad\mf Q\cdot\mf n=\psi\quad\mbox{on}\ \Sigma.$$ 
In addition, $\mf Q$ fulfills relations (\ref{g0.11}) and $\sup_Z|Q|<1$, according to our assumtions (\ref{g4.8}). Consequently, the preconditions of Theorem\,\ref{t1} and Corollary\,\ref{c1} are satisfied. The graph representation of the existing (and unique) stable $\ml H$-surface $\mf x\in\ml C_\mu(\Gamma,S,\overline Z)$ yields the desired solution of (\ref{g0.12-1}), (\ref{g0.12-2}).
\end{proof}

\end{document}